\documentclass[12pt]{article}

\usepackage[margin=1in]{geometry}         %nohead       % See geometry.pdf to learn the layout options. There are lots.

\usepackage{longtable}
\usepackage{graphicx}
\usepackage{caption,subcaption}
\usepackage{stmaryrd}
\usepackage{framed}
\usepackage{amssymb}
\usepackage{epstopdf}
\usepackage{amsmath,amsfonts,amssymb}
\usepackage{enumerate}
\usepackage{multirow}
\usepackage[color,notcite,notref]{showkeys}
\usepackage{fixltx2e}
\usepackage{bbm}
\usepackage{url}
\usepackage{cite}
\usepackage{fullpage}
\usepackage{fancyhdr}
\usepackage{color}
\usepackage{float}
\usepackage{soul}
\usepackage{graphicx}
\usepackage[doc]{optional}
\usepackage{enumitem}
\usepackage{xcolor}
\usepackage[toc,page]{appendix}
\usepackage{theorem}
\usepackage{array}
\usepackage{booktabs}
\usepackage{epsfig}
\usepackage{latexsym}
\definecolor{labelkey}{rgb}{1,1,1}
\definecolor{refkey}{rgb}{1,1,1}

\definecolor{myblue}{rgb}{.9, .9, 1}

\newtheorem{theorem}{Theorem}[section]
\newtheorem{lemma}[theorem]{Lemma}
\newtheorem{corollary}[theorem]{Corollary}
\newtheorem{proposition}[theorem]{Proposition}
\newtheorem{definition}[theorem]{Definition}

\theoremstyle{plain}{\theorembodyfont{\rmfamily}

\theoremstyle{plain}{\theorembodyfont{\rmfamily}
}
\theoremstyle{plain}{\theorembodyfont{\rmfamily}
}
\theoremstyle{plain}{\theorembodyfont{\rmfamily}
}
\theoremstyle{plain}{\theorembodyfont{\rmfamily}
\newtheorem{example}[theorem]{Example}}
\newtheorem{fact}[theorem]{Fact}
\theoremstyle{plain}{\theorembodyfont{\rmfamily}
\newtheorem{remark}[theorem]{Remark}}

\theoremstyle{plain}{\theorembodyfont{\rmfamily}
\newtheorem{customC}{Algorithm}
\newenvironment{Calg}[1]
  {\renewcommand\thecustomC{#1}\customC}
  {\endcustomC}

\theoremstyle{plain}{\theorembodyfont{\rmfamily}

\theoremstyle{plain}{\theorembodyfont{\rmfamily}
\newtheorem{customL}{Linesearch}
\newenvironment{linesr}[1]
  {\renewcommand\thecustomL{#1}\customL}
  {\endcustomL}

\def\proof{\noindent{\it Proof}. \ignorespaces}
\def\endproof{\ensuremath{\hfill \quad \blacksquare}}

\newcommand{\scal}[2]{\left\langle{#1},{#2}  \right\rangle}

\newcommand{\dom}{\ensuremath{\operatorname{dom}}}

\def\RR{{\mathbb{R}}}

\def\NN{{\mathbb{N}}}

\def\ox{\overline{x}}

\def\dd{\delta}

\newcommand{\la}{\langle}
\newcommand{\ra}{\rangle}

\newcommand{\disp}{\displaystyle}

\newcommand{\nexto}{\kern -0.54em}

\newcommand{\dZ}{{\cal Z \kern -0.7em Z}}
\newcommand{\dC}{{\rm\hbox{C \kern-0.8em\raise0.2ex\hbox{\vrule height5.4pt width0.7pt}}}}
\newcommand{\dQ}{{\rm\hbox{Q \kern-0.85em\raise0.25ex\hbox{\vrule height5.4pt width0.7pt}}}}

\newenvironment{retraitsimple}{\begin{list}{--~}{
 \topsep=0.3ex \itemsep=0.3ex \labelsep=0em \parsep=0em
 \listparindent=1em \itemindent=0em
 \settowidth{\labelwidth}{--~} \leftmargin=\labelwidth
}}{\end{list}}

\begin{document}
\title{{A projection algorithm for non-monotone variational inequalities}}
\author{Regina S. Burachik \footnote{School of Information Technology and Mathematical Sciences, University of South Australia, Mawson Lakes, SA 5095,
Australia.
 E-mail: \texttt{regina.burachik@unisa.edu.au}} \and R. D\'iaz Mill\'an\footnote{Federal Institute of Goi\'as, Rua 75, No.46, Centro. CEP: 74055-110, Goi\^ania, Brazil.
              E-mail: \texttt{rdiazmillan@gmail.com}}}

\maketitle

\vskip 8mm

\begin{abstract}
  We introduce  a projection-type algorithm for solving the variational inequality pro\-blem for point-to-set operators, and study its convergence properties. No monotoni\-city assumption is used in our analysis. The operator defining the
  problem is only assumed to be continuous in the point-to-set sense,
  i.e., inner- and outer-semicontinuous. Additionally, we assume non-emptiness of the so-called dual solution set. We prove that the
  whole sequence of iterates converges to a solution of the variational
  inequality. Moreover, we provide numerical experiments illustrating
  the behaviour of our iterates. Through several examples, we provide a comparison with a recent similar
  algorithm.

\bigskip

\medskip

\end{abstract}
%%%%%%%%%%%%%%%%%%%%%%%%%%%%%%%%%%%%%%%%%%%%%%%%%

\section{Introduction}
Variational inequalities were
introduced in 1966 by Hartman
and Stampacchia (see \cite{hart-stamp}), and have numerous important
applications in physics, engineering, economics, and optimization
theory (see, e.g., \cite{hart-stamp, stamp-kinder,191,chaos} and the
references therein).
The variational inequality problem for a point-to-set operator $T:\dom(T)\subseteq \RR^n\ \rightrightarrows \RR^n$ and a nonempty closed and convex
set $C \subset \dom(T)$, is stated as
  \begin{equation}\label{prob}
 \mbox{Find} \ \ x_*\in C  \ \  \mbox{such that } \exists u_*\in T(x_*), \ \mbox{with } \ \la u_*, x-x_*\ra\geq 0,  \ \  \ \forall x\in C.
   \end{equation}
By $S_*$ we denote the solution set of Problem
\eqref{prob}. This problem may be studied via its so-called {\em dual
  formulation}, which is stated as
 \begin{equation}\label{dual}
  \mbox{Find} \ \ x_*\in C  \ \  \mbox{such that } \forall x\in C, \mbox{and } \forall u\in T(x), \ \  \la u,x-x_*\ra\ge 0.
 \end{equation}
We denote the solution set of Problem (\ref{dual}) by $S_0$.  It is
easy to see that $S_0$ is closed and convex. However, in
general, $S_*$ is not. Most of the convergence analysis available for
variational inequalities relies on some kind of monotonicity
assumption. Namely, in the case of $T$ being a point-to-set, i.e., $T(x)$ is a subset of $\RR^n$, a standard assumption for analyzing Problem \eqref{prob} is either: maximal monotonicity (see e.g., \cite{bena,re-yu}), pseudo-monotonicity (see e.g., \cite{ceng,neto}) or  quasi-monotonicity \cite{nils}. In the point-to-point case, continuity of $T$, as well as $S_*\not=\emptyset$ are standard assumptions for
analyzing \eqref{prob}, (see e.g. \cite{liu,saew,mai,yu-re-hu}). In view of its wide range of applications, it is imperative to consider general versions of \eqref{prob}, which relax the standard assumptions mentioned above.

For solving variational inequalities, projection-type methods (see, e.g., \cite{rei-yun,pang,ldmu,ye}) are very popular because the iteration can be performed cheaply, when the set $C$ has a simple structure (e.g., a ball or a polyhedral set). The other methods of choice for variational inequalities may be cast as proximal-like or interior point methods (see e.g.,\cite{bur-joy, bur-sus, xuli,fuqua, bu-ju, nils}). The prox-like methods, however, may result in iterations which are as complex as solving the original problem and usually involve some kind of monotonicity assumption on $T$. Even though interior point methods may be more practical than prox-like methods in some instances, they may have, as stated in \cite{yacek} a ``lack of an efficient warmstarting scheme which would enable the use of information from a previous solution of a similar problem".

%For example, in \cite{chaos}, Problem \eqref{prob} is used to formulate optimal control problems in which
%
%we would like to address in
%
%Some nice examples can be found in

In the present paper, we devise a projection-type method for point-to-set variational inequalities and establish convergence
to a solution of Problem \eqref{prob}
under three basic assumptions: (i) non emptiness of the set $S_0$, (ii) local boundedness of the operator over the feasible set $C$, and
(iii) a suitable concept of continuity for point-to-set
operators. The concept of continuity we use in
(iii) may be found in for example \cite{bur-iu}, and formally stated in
Definition \ref{def:cont}.

If $T$ is point-to-point and maximally monotone, then it will automatically satisfy
assumption (i) through (iii) whenever $S_*\not=\emptyset$. Thus, in these cases whenever the problem presents a solution, our analysis is valid. When the operator is point-to-set and maximal monotone, with $ int \ co (\dom(T))\neq \emptyset$, the assumption (ii) is satisfied by Rockafellar's Theorem, see \cite{rocka}.

In Proposition \ref{soin} we prove that, when $C$ is contained in the domain of $T$, assumption (iii) implies that
$S_0\subseteq S_*$. So, the existence of solutions of \eqref{dual}
implies $S_*\not=\emptyset$. For the inclusion $S_*\subseteq S_0$ to hold, an
extra condition, such as  pseudo monotonicity,  is needed (see \cite[Lemma~1]{konnov1}).

Assumption (i) has been used in
\cite{ye} for variational inequalities with a
point-to-point operator, as well as in \cite{jean, bui} for the equilibrium
problem (i.e., for the point-to-point case). As far as we know, assumptions (i) and (iii) haven't been used for the point-to-set case. Condition (i) along with $S_0=S_*$ is a well-known example of an assumption that does not involve a monotonicity requirement on T. For example, in \cite{konnov-97, yu-re-hu, svaiter}, this assumption is used for the point-to-point case.

The algorithm considered in \cite{ye} uses assumption (i) for the point-to-point case. The difference between their method and ours can be explained as follows. In \cite{ye}, the current point $x^k$ is projected onto a subset containing the solution set. At each iteration of our algorithm, we project the same point $x^0$ onto a set, which is strictly smaller than the one used in \cite{ye}. The way of defining the iterates in \cite{ye}, allows for the use of Fej\'er convergence, which is a classical tool for this kind of projection algorithm. In the present paper, instead of projecting the current iterate, we project the fixed point $x^0$, therefore we do not use of Fej\'er convergence, as in \cite{ye}. Moreover, if our sequence does not have finite termination, the limit may be characterized as the closest point to the initial iterate $x^0$ in the set $\bar{co}(\tilde{S}_*)$ (see Proposition \ref{pertenece} and Theorem \ref{acumul1}). Specifically, $x^k \rightarrow \bar{x}$ where $\bar{x}=P_{\bar{co}(\tilde{S}_*)}(x^0)$. Unlike \cite{ye}, our method may be applied to point-to-set monotone, pseudo- or quasi-monotone variational inequalities, such as those in \cite{re-yu, ceng, neto, nils}.

Other projection algorithms for solving variational inequalities are found in \cite{ldmu,cegire,fache,fache2}. The main difference between our algorithm and  \cite{ldmu,cegire,fache,fache2} lies in the structure of the problem and the techniques used proving the convergence. For instance, \cite{ldmu,cegire} considers a Lipschitz continuous point-to-point operator, making the analysis of these algorithms substantially different from ours. In addition, in \cite{ldmu,cegire}, no linesearch is considered. References \cite{fache, fache2} do use a linesearch, but it is different from the one we use. Moreover, another difference between our method and the ones in \cite{fache, fache2} is the constrained set $C$, which is assumed to be compact. The additional assumption $S_0=S_*$, requires a different analysis of convergence.

This paper is organized as follows. In section \ref{preli} we introduce the notation, definitions, and some useful results. In Section \ref{linesearch} we define the linesearch and algorithm. Section \ref{conver} provides the convergence analysis of our algorithm. Section \ref{numer} presents numerical examples and comparisons. Finally, Section \ref{conclu} contains our conclusions and open problems.

\section{Preliminaries}\label{preli}

In this section, we introduce some known definitions, facts and properties necessary in the sequel. First, we fix the notation and recall some definitions. The inner product in $\RR^n$ is denoted by $\la \cdot , \cdot \ra$ and its norm by $\|\cdot\|$. For a nonempty, convex and closed subset, $C\subseteq\RR^n$, the {\em orthogonal projection of $x$ onto $C$} will be denoted by $P_C(x)$, and defined as the unique point in $C$ such that $\| P_C(x)-x\| \le \|y-x\|$ for all $y\in C$. Being $(x^k)_{k\in\NN}$ a sequence in $\RR^n$, we denote by $Cl(x^k)_{k\in \NN}$ the set of its cluster points. For the point-to-set operator $T$, we define the {\em domain of $T$} as $\dom(T):=\{x\in \RR^n\::\: T(x)\not=\emptyset\}$, and the {\em graph of $T$} as $Gr(T):=\{(x,u)\in \RR^n\times\RR^n\::\: u\in T(x)\}$.

 We begin with a concept of continuity for point-to-set operators. Our definitions are standard and taken from \cite{bur-iu}.

\begin{definition}\label{def:cont}
Let $T:\dom(T)\subset \RR^n \rightrightarrows \RR^n$ be a point-to-set mapping.
\begin{itemize}
\item [(a)] \label{osc} $T$ is said to be {\em outer-semicontinuous} ({\it OSC}), if and only if,  the graph of $T$ is closed.
\item [(b)] \label{isc} $T$ is said to be {\em inner-semicontinuous} ({\it ISC}) at $x\in \dom(T)$, if and only if, for any $y\in T(x)$ and for any sequence $(x^k)_{k\in\NN}\subset \dom (T)$, such that $x^k \to x$; there exist a sequence $(y^k)_{k\in\NN}$, such that $y^k\in T(x^k)$ for all $k\in\NN$ and $y^k \to y$. $T$ is {\it ISC} if it is {\it ISC} for every $x \in \dom(T)$.
\item [(c)] $T$ is said to be {\em upper-semicontinuous} ({\it USC}) at $x\in \dom(T)$, if and only if, for all  open $W \subset \RR^n$, such that $W \supset T(x)$; there exists a neighborhood $U$ of $x$, such that $T(x' )\subset W$ for all $x' \in U$. $T$ is {\it OSC} if it is {\it OSC} for every $x \in \dom(T)$.
\item [(d)] $T$ is said to be {\em continuous} if it is {\it ISC} and {\it OSC}.
\item [(e)] $T$ is said to be {\em locally bounded} at $x \in \dom(T)$ if there exist a neighborhood $U$ of $x$ such that
$$T(U):=\bigcup \{T(y):y\in U\}$$
is a bounded set. It is called {\em locally bounded} on $C\subseteq \RR^n$ if this holds at every $x\in C$.
 \end{itemize}
\end{definition}
\begin{remark}\label{localbound}
 Note that in finite dimensional spaces, locally bounded is equivalent to mapping bounded sets into bounded sets, for more details, see \cite{rockafe}, Proposition 5.15.
  \end{remark}
 The following Proposition relates the sets $S_0$ and $S_*$.
\begin{proposition}\label{soin}
Let the point-to-set mapping $T:\dom(T)\subset \RR^n\rightrightarrows \RR^n$ be {\it ISC}, then $S_0\subseteq S_*$.
\end{proposition}
\begin{proof}
Take $x^* \in S_0\cap \dom(T)$. Then for all $(y,v)\in Gr(T)$ such that $y\in C\cap \dom(T)$, we have $\la v,y-x^*\ra\geq 0$. Now, for all $\alpha \in (0,1)$ by convexity of $C$ we have that $y_{\alpha}:=(1-\alpha)x^*+\alpha y \in C$ for all $y\in C\cap \dom(T)$. Taking $u_* \in T(x^*)$, there exist $v_*^{\alpha}\in T(y_{\alpha})$ such that $\lim_{\alpha \to 0}v_*^{\alpha}=u_*$. Now, using that $x^*\in S_0$, we obtain:
$$
0\leq \la v_*^{\alpha},y_{\alpha}-x^*\ra =\alpha\la  v_*^{\alpha},y-x^*\ra.
$$
 Dividing by $\alpha >0$ and taking the limit when $\alpha$ goes to zero, we establish that $\la u_*,y-x^*\ra\geq 0$, for all $y\in C$. Thus $x^* \in S_*$.
\end{proof}
\begin{remark}
The inclusion $S_0\subseteq S_*$ has been established in \cite{konnov-book, konnov1, tan}. These papers assume $T$ to be {\it USC} and such that $T(x)$ is
 compact for all $x\in \dom(T)$. More precisely, if $T$ has closed images, then upper-semicontinuity implies outer-semicontinuity (see \cite[Proposition 2.5.12 (b)(c)]{bur-iu}) and therefore our analysis includes the cases considered in \cite{konnov-book, konnov1, tan}. As far as we know, Proposition \ref{soin} is new for $T$ point-to-set and {\it ISC}.
 An example showing an operator $T$ which is {\it OSC} but not {\it USC} is to be found in \cite[Example 2.5.8]{bur-iu}. While {\it upper-semicontinuity} can be seen as a natural
extension of the point-to-point continuity, it cannot express properly continuity of mappings in which $T(x)$ in unbounded (see, e.g., \cite[Example 2.5.8]{bur-iu}). Hence our
choice of {\it OSC} over {\it USC}. In Example \ref{ptso}, we implement our algorithm for a point-to-set operator which is not {\it USC} but is continuous (and hence {\it ISC}) in the sense of Definition
 \ref{def:cont}.
\end{remark}

 Now, we present some important facts on orthogonal projections, which proves useful, when defining the Linesearch presented in Section \ref{linesearch}.
\begin{fact}\label{proj}
Let $C\subseteq \RR^n$ be closed and convex. For all $x,y\in \RR^n$ and all $z\in C $, the following holds:
\begin{enumerate}
\item\label{proj-i} $\|P_C(x)-P_C(y)\|^2 \leq \|x-y\|^2-\|(x-P_C(x))-(y-P_C(y))\|^2.$
\item\label{proj-ii} $\la x-P_C(x),z-P_C(x)\ra \leq 0.$
\end{enumerate}
\end{fact}
\proof
See \cite[Proposition 4.8 and Theorem 3.14]{librobauch}.
\endproof
\begin{remark}\label{procon}
  By Fact \ref{proj}\ref{proj-i}, the map $P_C$ is Lipschitz continuous, and hence it maps bounded sets into bounded sets.
\end{remark}
\subsection{Some useful results}
The following three results are standard in the literature of variational inequalities. Here, for the convenience of the reader, we include their proofs.

The next property shall be used for the stopping criteria of the algorithm as well as in the finite termination of the {\bf Linesearch \ref{feasible}}.

\begin{proposition}\cite[Proposition 1.5.8]{pang}\label{parada}
Given $T:\dom(T)\subseteq \RR^n\ \rightrightarrows \RR^n$ and $C\subset \dom(T)\subset \RR^n$. If for some $u\in T(x)$ and $\beta>0$,
$x=P_C(x-\beta u)$, then $x\in S_*$.
\end{proposition}
\proof
Due to the Fact \ref{proj}\ref{proj-ii}, we have $\la x-\beta u - P_C(x-\beta u),y-P_C(x-\beta u)\ra\leq 0$ for all $y\in C$, using that $x=P_C(x-\beta u)$ as well as $\beta >0$, it follows $\la u, y-x\ra\geq 0$ for all $y\in C$. Proving that $x \in S_*$.
\endproof

Now we show a lemma which ensures that the hyperplanes used in the algorithm contains the solution set of Problem \eqref{dual}.

\begin{lemma}\cite[Lemma 2.17]{yu-re-hu}\label{propseq}
For any $z\in C$ and $u \in T(x)$, define $H(z,u) := \big\{ y\in \RR^n : \la
u,y-z\ra\le 0\big \}$.
Then, $S_0\subseteq H(z,u)$.
\end{lemma}
\proof
For $x_{*}\in S_0$ we have that, $\la u,x-x_*\ra\leq 0$  for all $(x,u)\in Gr(T)$ and with $x\in C$, then $x_{*}\in H(z,u)$.
\endproof

The following lemma is crucial, when proving that the hyperplanes used in the algorithm, separate the current iterate from the solution set.
\begin{lemma}\label{separa}
Let $C\subset \RR^n$ be a  closed, convex and nonempty set. Take $x\in C$ and $z=P_C(x-\beta u)$, with $\beta >0$ and $u\in \RR^n$. Assume that:
\begin{itemize}
\item[(i)] $\ox=\alpha z+(1-\alpha)$, with $\alpha \in (0,1)$.
\item [(ii)] $(\ox,\overline{u}), (x,u)\in Gr(T)$.
\item [(iii)] $\la \overline{u},x-z\ra\geq \delta \la u,x-z\ra$.
\end{itemize}
Then with $H(x,u)$ as in Lemma \ref{propseq}, $x\in H(\ox,\overline{u})$ implies that $x\in S_*$.
\end{lemma}
\proof
As $x\in H(\ox, \overline{u})$, we have that $\la \overline{u}, x-\ox\ra\le 0$. Using the Fact  \ref{proj} \ref{proj-ii}, we have
\begin{align}\label{deja ver}\nonumber
0\geq&\, \la \overline{u}, x-\ox\ra=\alpha\la
\overline{u},x-z\ra\ge\alpha\delta\big\la u, x -z\big \ra  \\  =&\,
\frac{\alpha}{\beta}\la z-(x-\beta  u ),x-z\ra+\frac{\alpha}{\beta}\|x-z\|^2 \nonumber \\ \geq &\, \frac{\alpha}{\beta}\delta\|x-z\|^2
\ge \, 0,
\end{align}
implying that $x=z$. By Proposition \ref{parada}, we conclude that $x\in S_*$.
\endproof

The next  result  will be used for proving the boundedness of the sequence generated by the algorithm and will play an important role for the convergence analysis presented in Section 4.

\begin{lemma}\cite[ Lemma 2.10]{yu-re-hu}\label{l:lim-2} Let $S$ be a nonempty, closed and convex
set. Take $x^0,x\in\RR^n$. Assume that $x^0\notin S$ and that
$S\subseteq W(x):= \{y\in \RR^n :\la y-x,x^0-x\ra\le0\}$. Then, $
x\in B[\tfrac{1}{2}(x^0+\ox),\tfrac{1}{2}\rho]$, where
$\ox=P_{S}(x^0)$ and $\rho={\rm dist}(x^0, S)=\|x_0-P_{S}(x_0)\|$.
\end{lemma}

\proof  First, since $S$ is convex and closed, $\ox=P_{S}(x^0)$ and $\rho={\rm dist}(x^0,S)$ are well-defined. Moreover, $S \subseteq W(x)$ implies that $\ox=P_{S}(x^0)\in W(x)$. Define $v:=\tfrac{1}{2}(x_0+\ox)$ and $r:=x^0-v=\tfrac{1}{2}(x^0-\ox)$, then $\ox-v=-r$ and $\|r\|=\tfrac{1}{2}\|x^0-\ox\|=\tfrac{1}{2}\rho$. Since $\bar{x}\in W(x)$, we write
\begin{equation*}
\begin{aligned}
0&\geq\la \ox-x,x^0-x\ra =\scal{\ox-v+v-x}{x^0-v+v-x}\\
&=\scal{-r+(v-x)}{r+(v-x)}=\|v-x\|^2-\|r\|^2.
\end{aligned}
\end{equation*}
This proves the result.
\endproof

The following proposition serves to show that the distance between consecutive iterates tends to zero. It is well-know, however hard to track down, for this reason, we include its proof here.
\begin{proposition}\label{wx}
Let $x^0, x \in \RR^n$ and $W(x)=\{y\in \RR^n :\la y-x,x^0-x\ra\le0\}$, then it holds that  $x=P_{W(x)}(x^0)$.
\end{proposition}
\proof
Since, $x\in W(x)$ and $P_{W(x)}(x^0)\in W(x)$ using Proposition \ref{proj} \ref{proj-ii} we have,
\begin{eqnarray}
\la P_{W(x)}(x^0)-x,x^0-x\ra\leq& 0\label{wx0} \\
\la P_{W(x)}(x^0)-x,P_{W(x)}(x^0)-x^0\ra\leq &0.\label{wx1}
\end{eqnarray}
Summing \eqref{wx0} and \eqref{wx1} we obtain $\| P_{W(x)}(x^0)-x\|^2\leq 0$, then $x=P_{W(x)}(x^0)$.
\endproof
\section{The linesearch and the algorithm}\label{linesearch}

Our linesearch is a modification of a search strategy first introduced in 1997, see \cite{iusem-be}. The authors of \cite{iusem-be} use the square of the norm on the right-hand
side of the inequality in the {\bf Linesearch \ref{feasible}} ({\bf F} stands for {\it feasible} direction method ). Later on, Konnov in \cite{konnov-97} uses a linesearch as the one
 we use below,  but for point-to-point mappings. Both use the assumption $S_0=S_*$.

\begin{center}\fbox{\begin{minipage}[b]{\textwidth}
\begin{linesr}{F}
{\rm(feasible direction)}
\label{feasible}
\medskip
{\bf Input:} $x\in C$,  $\beta>0$ and $\dd\in(0,1)$.

Set $\alpha\leftarrow 1$ and $\theta \in (0,1)$. Define $z=P_C(x-\beta u)$ with $u\in T(x)$

\begin{retraitsimple}
\item[] {\bf If } $\forall u_{\alpha} \in T\big(\alpha z+(1-\alpha) x\big)$,  $\la u_{\alpha}, x-z\ra
< \delta \la u, x-z\ra$ {\bf then} $\alpha\leftarrow\theta \alpha$, {\bf Else} Return $\alpha$.
\end{retraitsimple}
{\bf Output:} $(\alpha)$.
\end{linesr}\end{minipage}}\end{center}

Let $C$ be a convex and closed set. As mentioned in the Introduction, in our analysis we will use the following assumptions on $T$:
\begin{enumerate}[leftmargin=0.5in, label=({\bf A\arabic*})]
\item\label{a0} The feasible set $C$ is contained in the domain of $T$, i.e., $C\subset \dom(T)$.
\item\label{a1} $T$ continuous on $C$, in the sense of Definition \ref{def:cont}(d).
\item\label{a3} $T$ is locally bounded on $C$.
\item\label{a2}  The solution set $S_0$ of the Dual Problem (\ref{dual}) is not empty.
\end{enumerate}
The fact that the {\bf Linesearch \ref{feasible}} has finite termination (and hence, is well defined) is proved next.
\begin{lemma}\label{feasible-well}
Assume that \ref{a0} holds and $T$ is {\it ISC} at every point of $C$. If $x\in C$ and $x\notin S_*$, then {\bf Linesearch \ref{feasible}} stops after a finite number of steps.
\end{lemma}
\proof
Since $T$ is {\it ISC} at $x$, given $u\in T(x)$ and $y_{\alpha}\rightarrow x$, with $y_\alpha= \alpha z+(1-\alpha) x$ and $\alpha\in (0,1)$ there exist $v_{\alpha}\in T(y_{\alpha}): v_{\alpha}\rightarrow u$ when $\alpha \to 0$. Now, suppose that {\bf Linesearch \ref{feasible}} never stops, then we have:
\begin{equation}\label{delta}
\la v_{\alpha}, x-z\ra < \delta \la u, x-z\ra.
\end{equation}
Taking limits in \eqref{delta} when $\alpha \to 0$
$$\la u, x-z\ra\leq \delta\la u,x-z\ra \Leftrightarrow (1-\delta) \la u,x-z\ra \leq 0.$$
 Since $\delta \in (0,1)$,

 $$0\geq \la u, x-z\ra=\frac{1}{\beta}\big(\|x-z\|^2+\la z-(x-\beta u),x-z \ra \big),$$

\noindent using Fact \ref{proj} \ref{proj-ii} we get $\|x-z\|^2\leq \la(x-\beta u)-z,x-z \ra\leq 0$, which implies that $x=z$. Hence, $x\in S_*$ by Proposition \ref{parada}. This contradicts our assumption $x\notin S_*$. Thus, the well definition of {\bf Linesearch \ref{feasible}} follows.
\endproof

\begin{remark}\label{admis}
The implementation of the {\bf Linesearch \ref{feasible}} for point-to-set mappings might be a nontrivial task.  In Example \ref{ptso} we present an operator $T$, for which this implementation is possible.
\end{remark}

Recall from Section \ref{preli} that
\begin{equation}\label{hk}
H(z,v):=\big\{ y\in \RR^n : \la v,y-z\ra\le 0\big \}
\end{equation} and
\begin{equation}\label{wk}
W(x):=\big\{ y\in \RR^n : \la y-x,x^0-x\ra\le 0\big \}.
\end{equation}

These halfspaces (as well as their intersections) have been widely used in the literature, e.g., \cite{yu-re-hu, bui, ye, sva , yun-reinier-1}.

Now we describe the Algorithm.

\begin{center}\fbox{\begin{minipage}[b]{\textwidth}
\begin{Calg}{F}(Feasible direction algorithm)\label{A1}
Given $(\beta_k)_{k\in \NN}\subset[\check{\beta},\hat{\beta}]$ such that $0<\check{\beta}\le \hat{\beta}<+\infty$ and $\dd\in(0,1)$.
\item[ ]{\bf Initialization:} Take $x^0\in C$, define $\tilde{H}_0:=\RR^n$ and set $k\leftarrow 0$.

\item[ ]{\bf Step~1: } Set $z^k=P_C(x^k-\beta_k u^k)$ with $u^k\in T(x^k)$ and
\begin{equation}\label{vbar}
\alpha_k= {\bf Linesearch\; \ref{feasible}}\;(x^k, \beta_k,\dd),
\end{equation}
i.e., $(\alpha_k,z^k)$ satisfy
\begin{equation}\label{zk212*}
\left\{
\begin{aligned}
&\la \bar{u}^k,x^k-z^k\ra\geq\delta\la u^k,x^k-z^k\ra.
\end{aligned}\right.
\end{equation} with $\overline{u}^k\in T(\alpha_k z^k+(1-\alpha_k)x^k)$.
\item[ ]{\bf Step~2 (Stopping Criterion):} If  $z^k=x^k$  or $z^k=P_C(z^k-v^k)$ with $v^k\in T(z^k)$, then stop. Otherwise,
\item[ ]{\bf Step 3:} Set
\begin{subequations}
\begin{align}
\ox^k&:=\alpha_k z^k+(1-\alpha_k)x^k,{\label{xbar2}}\\
\tilde{H}_k&:=\tilde{H}_{k-1}\cap H(\overline{x}^k,\overline{u}^k),\label{htilde} \\
\text{and}\quad
x^{k+1}&:=P_{C \cap \tilde{H}_k\cap W(x^k)}(x^0);\label{P112}
\end{align}
\end{subequations}
\item[ ]{\bf Step~4:} If $x^{k+1}=x^k$, then stop. Otherwise, set $k\leftarrow k+1$ and go to {\bf Step~1}.
\end{Calg}\end{minipage}}\end{center}

\section{Convergence Analysis}\label{conver}
Our goal in this section  is to establish the convergence of the algorithm. First of all, let us see that the stopping criterion is well defined.
\begin{proposition}\label{stop}
If the {\bf Algorithm \ref{A1}} stops at {\bf Step 2}, then $x^k$  or $z^k$ are  solutions.
\end{proposition}
\proof
This is a direct consequence of the definition of $z^k$, $v^k$ and Proposition \ref{parada}.
\endproof
\begin{proposition}\label{H-separa-x12} Let $(\ox^k)_{k\in \NN}$, $(x^k)_{k\in\NN}$ and $(\overline{u}^k)_{k\in \NN}$ be sequences generated by {\bf Algorithm \ref{A1}}. If $x^k \in H(\ox^k,\overline{u}^k)$ (see \eqref{hk}), then $x^k\in S_*$.
\end{proposition}
\proof
Follows by applying Lemma \ref{separa} for $\alpha=\alpha_k$, $x=x^k$, $\ox=\ox^k$, $u=u^k$ and $\overline{u}=\overline{u}^k$, using {\bf Linesearch \ref{feasible}}.
\endproof

\medskip
As a direct consequence of  {\bf Linesearch \ref{feasible}}, we state the following remark, pointing out a useful algebraic property of the sequence generated by {\bf Algorithm~\ref{A1}}.

\begin{remark}\label{useful}
 Let $(x^k)_{k\in \NN}$ and $(\alpha_k)_{k\in \NN}$
be sequences generated by {\bf Algorithm \ref{A1}}, using \eqref{deja ver}, we get
\begin{equation}\label{d:useful}
\forall k\in\NN:\quad\la \overline{u}^k,x^{k}-\ox^{k} \ra  \geq
\frac{\alpha_k}{\hat{\beta}}\delta\|x^k-z^k\|^2.
\end{equation}
\end{remark}

\begin{proposition}\label{stopp}
If $x^{k+1}=x^k$, then $x^k\in S_*$.
\end{proposition}
\proof
If  $x^{k+1}=P_{C\cap \tilde{H}_k\cap W(x^k)}(x^0)=x^k$, then  $x^k\in \tilde{H}_k$, which implies that $ x^k\in H(\ox^k,\overline{u}^k)$ ,  so that by Proposition \ref{H-separa-x12} we have that $x^k\in S_*$.
\endproof

If {\bf Algorithm \ref{A1}} stops in a finite number of iterations, then by Propositions \ref{stop} and \ref{stopp} the last iterate is a solution. Hence,
it is enough to establish convergence when the algorithm does not stop. Therefore, from now on, we suppose that the sequence $(x^k)_{k\in\NN}$ generated by the {\bf Algorithm F}, is infinite and  $x^k\notin S_*$ for all $k\in\NN$. The next result shows that the projection step is well-defined.
\begin{proposition}\label{pertenece}
Let $\tilde{H}_k$ be as in \eqref{htilde}, and define $\tilde{S}_*:=\cap_{k\in \NN}\tilde{H}_k\cap S_*$. Then, $\tilde{S}_*\subset H(\ox^k,\overline{u}^k)\cap W(x^k)$ for all $k \in \NN$ and $\tilde{S}_*  \neq \emptyset$.
\end{proposition}
\proof
By definition, we have that  $\tilde{S}_*\subset C\cap H(\ox^k,\overline{x}^k)$ for all $k\in \NN$. By induction we prove that $\tilde{S}_*\subset W(x^k)$ for all
 $k\in \NN$. For $k=0$ we have that $\tilde{S}_*\subset W(x^0)=\RR^n$, suppose that $\tilde{S}_* \subset W(x^k)$, then by the Fact \ref{proj} \ref{proj-ii},  we obtain
$\langle x_{*}-x^{k+1}\,,\, x^0-x^{k+1}\rangle\leq 0$, for all $x_{*}\in \tilde{S}_*$, which implies $x_{*}\in W(x^{k+1})$. Then, the result follow by induction. By Lemma
 \ref{propseq} we have that $S_0\subset H(\ox^k,\overline{u}^k)$ for all $k\in\NN$. By Assumption \ref{a1} and Proposition \ref{soin}, we deduce that  $S_0\subseteq S_*$,  hence
 $S_0 \subseteq \tilde{S}_*$ and by Assumption \ref{a2}, $\tilde{S}_* \neq \emptyset$.
\endproof

Now we prove the well definition of the iterates of {\bf Algorithm \ref{feasible}}.
\begin{proposition}
The sequence $(x^k)_{k\in\NN}$ is well defined and $(x^k)_{k\in\NN}\subset C$.
\end{proposition}
\proof
 By definition of the solution set, we have that $S_*\subset C$, then by Proposition \ref{pertenece},  for all $k\in\NN$ the closed and convex set $C\cap H(\ox^k,\overline{u}^k)\cap W(x^k)\neq \emptyset$, (note that $C$, $H(\ox^k,\overline{u}^k)$ and $W(x^k)$ are convex and closed sets). Therefore, the projection step  is well-defined. The fact that $x^k \in C$ for all $k\in\NN$ follows from the definition of the iterates in \eqref{P112} and the fact that $x^0\in C$.
\endproof

The next result proves the boundedness of the sequence generated by the algorithm.
\begin{proposition}\label{bounded}
The sequence generated by the algorithm satisfies that  $(x^k)_{k\in \NN}\subset B[\frac{1}{2}(x^0+\bar{x}),\frac{\rho}{2}]$, where $\bar{x}:=P_{S_0}(x^0)$ and $\rho=\|x^0-P_{S_0}(x^0)\|$. Therefore, the sequence $(x^k)_{k\in\NN}$ is bounded.
\end{proposition}
\proof
Since $S_0$ is a nonempty, convex and closed set  and $x^0 \notin S_0$, we are in the hypothesis of  Lemma \ref{l:lim-2}. Using this lemma with $S=S_0$ and $x=x^k$, the result follows.
\endproof

Next we show that the distance between consecutive iterates tends to zero.
\begin{proposition}\label{near}
The  sequence $(x^k)_{k\in \NN}$ satisfies that $\sum_{k=0}^{\infty} \|x^{k+1}-x^k\|^2<\infty$, hence $\lim_{k\to\infty}\|x^{k+1}-x^k\|=0$.
\end{proposition}
\proof
By Proposition \ref{wx}, for $x=x^k$, we have that $x^k=P_{W(x^k)}(x^0)$. Since $x^{k+1} \in W(x^k)$ then, by the Fact \ref{proj} \ref{proj-i}, we obtain that $0\leq \|x^{k+1}-x^k\|^2\leq \|x^{k+1}-x^0\|^2-\|x^k-x^0\|^2$. Summing this inequality from $k=0$ to $\infty$ and using the boundedness of the sequence $(x^k)_{k\in\NN}$, we obtain that $\sum_{k=0}^{\infty} \|x^{k+1}-x^k\|^2<\infty$. Therefore, $\lim_{k\to\infty}\|x^{k+1}-x^k\|=0$.
\endproof

The next result on the sequences generated by the {\bf Algorithm F} will be necessary for the convergence analysis. 
\begin{corollary}\label{bounded}
The sequences $(\overline{x}^k)_{k\in\NN}$, $(u^k)_{k\in\NN}$, $(\overline{u}^k)_{k\in\NN}$ and $(z^k)_{k\in\NN}$ generated by the {\bf Algorithm F} are bounded.
\end{corollary}
\proof
The boundedness of all sequences follows from assumption \ref{a3}, {\bf Algorithm F}, Remark \ref{localbound} and Remark \ref{procon}. 
\endproof

Now we present a key convergence result for our algorithm.
\begin{theorem}\label{acumul1}
Let $(x^k)_{k\in \NN}$ be the sequence generated by the {\bf Algorithm F}. Then  $Cl(x^k)_{k\in \NN}\subseteq \tilde{S}_* \subseteq S_*$.
\end{theorem}
\proof
First we prove that  $Cl(x^k)_{k\in \NN}\subseteq S_*$. Since $x^{k+1}\in H(\overline{x}^k, \overline{u}^k)$ for all $k\in \NN$ then by definition of $H(\ox,\overline{u})$ we obtain
 $\la \overline{u}^k,x^{k+1}-\overline{x}^k\ra \leq 0$.
Now,
$$0\geq \la \overline{u}^k,x^{k+1}-\overline{x}^k\ra=\la \overline{u}^k,x^{k+1}-x^k\ra+\alpha_k\la \overline{u}^k,x^{k}-z^k\ra.$$
 Using the same ideas as  in \eqref{deja ver} and Remark \ref{useful}, we have that
$$\frac{\alpha_k \delta}{\hat{\beta}}\|x^k-z^k\|^2 \leq \la
 \overline{u}^k,x^k-x^{k+1}\ra\leq \|\overline{u}^k\|\|x^k-x^{k+1}\|,$$
by Corollary \ref{bounded} the sequences $(\ox^k)_{k\in\NN}$ and $(\bar{u}^k)_{k\in\NN}$ are bounded. Passing to the limits for  $k\rightarrow \infty$ we obtain,
\begin{equation}\label{lim-0}
\lim_{k\to \infty} \alpha_{k}\|x^{k}-z^{k}\|=0.
\end{equation}
We take a subsequence $(i_k)_{k\in \NN}$, such that $(\alpha_{i_k})_{k\in\NN}$, $(\beta_{i_k})_{k\in\NN}$,  $(x^{i_k})_{k\in\NN}$, $(u^{i_k})_{k\in\NN}$
 and $(z^{i_k})_{k\in\NN}$ being convergent to $\tilde{\alpha}$, $\tilde{\beta}$, $\tilde{x}$, $\tilde{u}$ and $\tilde{z}$ respectively. This is possible by the boundedness
of all the sequences involved. Note that by Assumption \ref{a1}, we have that $Gr(T)$ is closed, and therefore $\tilde{u}\in T(\tilde{x})$. This leaves two cases:

{\bf Case~1:} $\disp\lim_{k\to \infty}\alpha_{i_k}=\tilde{\alpha}>0$. As consequence of \eqref{lim-0}, $
\lim_{k\to \infty}\|x^{i_k}-z^{i_k}\|=0$.
Using the continuity of the projection  $\disp \tilde{x}=\lim_{k\to
\infty}
x^{i_k}=\lim_{k\to
\infty}
z^{i_k}=P_C\big(\tilde{x}-\tilde{\beta}\tilde{u}\big)$.
Then, $\tilde{x}=P_C\big(\tilde{x}-\tilde{\beta}\tilde{u}\big)$,
and Proposition \ref{parada} implies that $\tilde{x}\in S_*$.

{\bf Case~2:} $\disp\lim_{k\to \infty}\alpha_{i_k}=\tilde{\alpha}=0$. Define
$\tilde{\alpha}_k=\frac{\alpha_{k}}{\theta}$. Then,
\begin{equation}\label{tildealphato0}\lim_{k\to\infty}\tilde{\alpha}_{i_k}=0.\end{equation} Define
$\tilde{y}^{k}:=\tilde{\alpha}_k{z}^{k}+(1-\tilde{\alpha}_k)x^{k}$. Hence,
\begin{equation}\label{ykgox}
\lim_{k\to\infty}\|x^{i_k}-\tilde{y}^{i_k}\|=0,
\end{equation}
which imply that the sequences $(x^{i_k})_{k\in\NN}$ and $(\tilde{y}^{i_k})_{k\in\NN}$ have the same cluster points.  From the definition of $\alpha_k$ in {\bf Algorithm~\ref{feasible}}, $\tilde{y}^{k}$ does not satisfy the inequality \eqref{zk212*}, that is, for all $v^k \in T(\tilde{y}^k)$ we have
\begin{equation}\label{conse}
\la {v}^{k} , x^{k} -
{z}^{k}\ra<\delta \la u^{k},
x^{k}-{z}^{k}\ra.
\end{equation}

As $\tilde{y}^{i_k}\to \tilde {x}$ we have by the continuity of $T$, that exists a sequence $v^{i_k}\in T(\tilde{y}^{i_k})$ that converges to $\tilde{u} \in T(\tilde{x})$. Taking this sequence and limits over the subsequence $(i_k)_{k\in\NN}$ in \eqref{conse} we have that  $\la \tilde{u}, \tilde{x}-\tilde{z}\ra \le \delta \la \tilde{u} , \tilde{x}-\tilde{z}\ra$. Then,
\begin{align*}\nonumber
0 & \, \ge (1-\delta) \big\la \tilde{u}, \tilde{x}-\tilde{z}\big\ra
=\frac{(1-\delta)}{\tilde{\beta}} \big\la
\tilde{x}-(\tilde{x}-\tilde{\beta} \tilde{u}), \tilde{x}
-\tilde{z}\ra \ge \frac{(1-\delta)}{\tilde{\beta}}\|\tilde{x}-\tilde{z}\|^2 \ge
0.\nonumber
\end{align*}
This means that  $\tilde{x}=\tilde{z}$, the continuity of the projection and Proposition \ref{parada} implies $\tilde{x}\in S_*$.

 We have proved that all cluster points belong to $S_*$. Now suppose that the sequence $(x^{n_k})_{k\in\NN}$ converges to $\bar{x} \notin H(\ox^{l_0}, \overline{u}^{l_0})$
for some $l_0 \in \NN$. As $H(\ox^{l_0},\overline{u}^{l_0})$ is closed, and for all $n_k >l_0$ using the definitions \eqref{P112}
 and \eqref{htilde}, we get that $x^{n_k}\in H(\ox^{l_0},\overline{u}^{l_0})$, which contradicts the fact that $\bar{x} \notin H(\ox^{l_0},\overline{u}^{l_0})$. This establish the result.

\endproof

\begin{theorem}
The sequence generated by the algorithm  converges to a point in the solution set $S_*$.
\end{theorem}
\proof
By Proposition \ref{pertenece}, the closure of the convex hull of $\tilde{S}_*$ ($\bar{co}(\tilde{S}_*)$), is contained in $W(x^k)$ for all $k\in\NN$ since  $W(x^k)$ is convex and closed. Since $\bar{co}(\tilde{S}_*)$ is a nonempty, convex and closed set  and $x^0 \notin \bar{co}(\tilde{S}_*) $, we may  apply Lemma \ref{l:lim-2} with $S=\bar{co}(\tilde{S}_*)$ and $x=x^k$. Hence, we have that $(x^k)_{k \in \NN} \subset B[\frac{x^0+\bar{x}}{2},\frac{\rho}{2}]$, where $\bar{x}=P_{\bar{co}(\tilde{S}_*)}(x^0)$ and $\rho=\|x^0-\bar{x}\|$. By Theorem \ref{acumul1}, all cluster points of the sequence belong to $\bar{co}(\tilde{S}_*)$. On the other hand, by the definition of $\overline{x}$ and $\rho$ we have $B[\frac{x^0+\bar{x}}{2},\frac{\rho}{2}]\cap \bar{co}(\tilde{S}_*)=\{\bar{x}\}$. This implies that $Cl(x^k)_{k\in \NN}=\{\ox\}$, therefore the sequence has only one cluster point and hence converges to the cluster point $\ox$. By Theorem \ref{acumul1}, we conclude that $\ox \in \tilde{S}_*\subseteq S_*$.
\endproof

\section{Numerical experiments}\label{numer}
In this section we show some numerical experiments to test  {\bf Algorithm \ref{A1}} and compare it with \cite[Algorithm 2.1]{ye}. We use MATLAB version R2015b on a PC  with Intel(R)
 Core(TM) i5-4570 CPU 3.20GHz and Windows 7 Enterprise, Service Pack 1. For the calculation of the projection step we use the Quadratic Programming (quadprog) tool.
In Examples 5.1 and 5.2 we use the stopping criterion $\|x^k-z^k\|^2\leq 10^{-8}$, with $x^k$ and $z^k$ generated by the algorithm, $\delta =0.01$,  $\beta_k=1$
for all $k\in\NN$, $\theta=0.5$. For ``$x^0$" we denote the initial point, {\it `` iter"} denotes the number of iteration of the algorithm, {\it ``nT"} denotes the number of
evaluations of the operator $T$. In Example 5.1  and \ref{ptso} {\it ``sol"} denotes the point at which the algorithm stops.  In Example 5.3  we use
 $\theta=0.25$ and tolerance $\|x^k-z^k\|^2\leq 10^{-4}$. In Example \ref{ptso} we use $\theta=\delta=0.5$ and $\beta_k=1$ for all $k\in\NN$,
 the tolerance used was  $\|x^k-z^k\|^2\leq 10^{-80}$. \\

\begin{example}\cite{ye,had}\label{haj}
Let $C=[0,1]\times[0,1]$ and $t=(x_1 +\sqrt{x_1^2+4x_2})/2$. We consider Problem \eqref{prob} with the operator $T:C\rightarrow \RR^2$ defined as:
$$
T(x_1,x_2)=(-t/(1+t),-1/(1+t)).
$$
\end{example}
This example was introduced by Hadjisavvas and Schaible in \cite{had} and was used in \cite{ye}.
 The operator $T$ in this example is quasimonotone (i.e., for all $(x,u),(y,v)\in Gr(T)$  we have that  $\la u,y-x\ra>0$ implies $\la v,y-x\ra>0$).
The solution set is $S_*=S_0=(1,1)$. The results are listed in {\bf Table 1}.\\

\begin{tabular}{ |p{2cm}| p{1.5cm} p{1.8cm} p{1.5cm} |p{1.5cm} p{1.8cm} p{1.5cm}|  }
 \hline
 \multicolumn{7}{|c|}{{\bf Table 1.} Results for example \ref{haj}.} \\
 \hline
 & {\bf Alg {\bf F}}  & &   &{\bf Alg 2.1} &in \cite{ye}&  \\
 \hline
 $x^0$   & iter(nT)   & CPU time  & sol & iter(nT) & CPU time & sol \\
 \hline
 (0,1)& 1(3) &0.249602& (1,1) & 3(2) &0.312002& (1,1)\\
 (0,0)& 1(3)& 0.234001 & (1,1)& 50(406)& 0.561604& (1,1)\\
 (1,0)& 2(4) & 0.265202& (1,1)& 71(331) & 0.561604& (1,0.999)  \\
 (0.5,0.5)& 0(2) & 0.0156001& (1,1) &1(2) & 0.234001& (1,1) \\
 (0.2,0.7)& 1(3) &0.249602 & (1,1) & 2(3) & 0.280802 & (1,1) \\
 (0.1,0.7)& 1(3) & 0.249602&(1,1) & 2(3)& 0.296402 & (1,1) \\
 \hline
\end{tabular}
\vspace{0.7cm}

The following example with $n=1$,  $\rho(x)=\rho_1(x)=\|x\|^2$,  and $a=1$, is \cite[Example 4.2]{ye}.
\begin{example}\label{myexa}
Let $C=[-a,a]^n$ with $a>0$, an consider $T:C\rightarrow \RR^n$ defined as $T(x)=(\rho_1(x), \rho_2(x), \cdots , \rho_n(x))$ where, for all $i=1,\cdots, n$,  $\rho_i :\RR^n \rightarrow \RR_+$ is a continuous function satisfying $\rho_i(x)=0$ , iff, $x=0$. Notice that $S_0=-a(1,1, \cdots, 1)$ and $S_*=S_0 \cup (0,0,\cdots,0)$. In this case, $S_0\neq S_*$ and since $T$ is continuous we may apply {\bf Algorithm \ref{A1}} to find the solution. See the results for $\rho(x)=\rho_i(x)=\|x\|^2$ and $\rho(x)=\rho_i(x)=\|x\|$ for all $i=1,\cdots, n$, and $a=1$, in {\bf Table 2}. In the first two rows of {\bf Table 2}, we note that the algorithm stopped at a point close to $(0,0,\cdots,0)$,  because the stopping criterion $\|x^k-z^k\|^2\leq 10^{-8}$ was satisfied. A similar (rather inaccurate) convergence result is observed for Algorithm 2.1 in \cite{ye}.
\end{example}

\begin{tabular}{|p{0.8cm} p{0.4cm} p{2.4cm}| p{1.3cm} p{1.5cm} p{1.5cm}|p{1.5cm} p{1.5cm} p{1.5cm}|}
 \hline
 \multicolumn{9}{|c|}{{\bf Table 2} Results for Example \ref{myexa}.}\\
 \hline
  & & & {\bf Alg {\bf F}}  & & & {\bf Alg 2.1}& in \cite{ye} & \\
 \hline
 $\rho$ &$n$ & $x^0$   & iter(nT)   & CPU time & sol &  iter(nT) & CPU time & sol  \\
 \hline
 $\|\cdot\|^2$& 1 & 0.1& 88(178) &0.608404& 0.0099 & 512(2797) &1.79401 & 0.0099\\
 $\|\cdot\|^2$& 1 &0.5& 94(190)& 0.592804 & 0.0099& 962(6428)& 3.05762& 0.01\\
 $\|\cdot\|^2$& 1 & -0.5& 2(8) & 0.34375 & -1& 2(7) & 0.296402 & -1 \\
 $\|\cdot\|$&5 &$10^{-3}(1,...,1)$& 7(23) & 0.4375&-(1,...,1)&7(22) & 0.46875 &-(1,...,1)\\
 $\|\cdot\|$&50 &$-0.1(1,...,1)$& 2(8) &0.375&-(1,...,1)& 2(7) & 0.2968 &-(1,...,1)\\
 $\|\cdot\|$&100 &$-0.1^{3}(1,...,1)$& 3(11) & 0.3906& -(1,...,1)& 3(10)& 0.4843 &-(1,...,1)\\
 \hline
\end{tabular}
\vspace{0.7cm}

The following example is \cite[Example 4.3]{ye}.
\begin{example}\label{linear}
Consider the feasible set $C=\{x\in \RR^5: x_i\geq 0, \ \ i=1,2,\cdots,5, \ \ \sum_{i=1}^5 x_i=a\}$ where $a>0$. The problem,
\begin{eqnarray}\label{mini}
&\min F(x) \\
&s.t.\ \  x\in C,
\end{eqnarray}
where $$F(x)=\frac{\tfrac{1}{2} \la Hx,x\ra+\la q,x\ra+1}{\sum_{i=1}^{5} x_i}$$ with $H$ being a positive diagonal matrix, with the same random element $0.1\leq h \leq 1.6$ in
 the diagonal and $q=(-1,-1,\cdots, -1)$. Note that $F$ is a smooth quasiconvex function then attains its
minimum value on a compact set $C$. This problem may be modelled as Problem \ref{prob}. $T$ is a point-to-point operator defined by $T=\nabla F$. Note that $T(x)=\Big(\frac{\partial F(x)}{\partial x_1}, \cdots, \frac{\partial F(x)}{\partial x_5}\Big)$,
with $$\frac{\partial F(x)}{\partial x_i}=\frac{hx_i\sum_{i=1}^{5} x_i-\tfrac{1}{2}h\sum_{i=1}^{5} x_i^2-1}{(\sum_{i=1}^{5} x_i)^2}.$$ For this example, we have a quasimonotone variational inequality with $S_0=\{\tfrac{1}{5}(a, \cdots,a)\}$. Some values for $\delta$ are tested for better comparison with  \cite[Algorithm 2.1]{ye}. See the results in {\bf Table 3}.
\end{example}
\begin{tabular}{|p{2cm} p{1cm} p{1cm}| p{2cm} p{2cm} |p{2cm} p{2cm}|}
 \hline
 \multicolumn{7}{|c|}{{\bf Table 3} Results for Example \ref{linear}.}\\
 \hline
  & & & {\bf Alg {\bf F}}  & & {\bf Alg 2.1} in & \cite{ye} \\
 \hline
 $x^0$ &$\delta$ & $a$   & iter(nT)   & CPU time  &  iter(nT) & CPU time  \\
 \hline
 $(0,0,5,0,0)$& 0.01& 5& 22(46) &0.218401& 567(2269) &7.22285\\
 $(0,2,0,2,1)$& 0.01 &5& 36(74)& 0.312002& 509(2546)& 6.81724\\
 $(0,0,5,0,0)$& 0.5 & 5& 14(30) & 0.156001& 22(45) & 0.249602  \\
 $(0,2,0,2,1)$&0.5 &$5$& 42(86) & 0.374402 & 21(43) & 0.218401 \\
 $(1,1,1,1,6)$&0.01&$10$& 94(190) &0.639604& 439(1757) & 5.25723 \\
 $(1,1,6,1,1)$&0.01&$10$& 101(204) & 0.686404& 482(1929)& 6.00604\\
 $(1,1,1,1,6)$&0.99&$10$& 712(2138) & 4.32123 & 40(81) & 0.358802 \\
 $(1,1,6,1,1)$&0.99&$10$& 846(2540) & 5.22603& 61(123)& 0.514803\\
 \hline
\end{tabular}
\\
\\

The following is an example with $T$  be point-to-set and continuous. This example is inspired by \cite [Example 2.5.8]{bur-iu}.
\begin{example}\label{ptso}
Let the point-to-set operator $T:\RR^2\rightrightarrows \RR^2$, be defined by $$T(x,\theta):=\{t(\cos(\theta),\sin(\theta)): t\geq x\},$$ and the
set $C=\{(x,\theta): x\geq 0, \  \theta\in[0,\pi/2]\}$. Consider Problem \eqref{prob} for $T$ and $C$.

It may be shown that the operator $T$ is continuous, but not {\it USC}. Since $(0,0)\in T(0,\theta)$ for all $\theta\in [0,\pi/2]$ we have the solution set
$S_*=\{(0,\theta): \theta \in [0,\pi/2]\}$. It may also be shown easily that $S_0=\{(0,0)\}$. In this example, we perform {\bf Step 1} as follows. Given $x^k=(t_k,\theta_k)$, take
$u^k=t_k(\cos(\theta_k),\sin(\theta_k))$. Our numerical results are reported in {\bf Table 4} below.
\end{example}

\begin{tabular}{|p{2cm} | p{2cm} p{2cm} p{3.5cm}|}
 \hline
 \multicolumn{4}{|c|}{{\bf Table 4} Results for Example \ref{ptso}.}\\
 \hline
  &  & {\bf Alg {\bf F}}  & \\
 \hline
 $x^0$      & iter(nT)   & CPU time  &  sol   \\
 \hline
$(1,\pi/2)$  & 7(16) &0.296402& $(0,0)$ \\
$(0.5,\pi/3)$ & 145(292)& 1.09201& $(0,0)$\\
$(0.1,\pi/2)$ &378(758)& 2.77682& $(0,0)$\\
$(100,\pi/2)$ & 6(15) & 0.327602& $(0,0)$  \\
$(0.1,\pi/10)$  & 89(180) & 0.702005 & $(0,0)$  \\
 $(1,\pi/100)$ &  7(16) &0.312002& $(0,0)$ \\
$(20,\pi/6)$ & 3(8) &0.280802& $(0,1.7526*10^{-16})$ \\
$(10,\pi/4)$ & 3(8) &0.312002& $(0,8.4431*10^{-12})$ \\
$(1500,\pi/8)$ & 5(12) &0.296402& $(0,4.3692*10^{-9})$ \\
 \hline
\end{tabular}
\\
\\

\begin{remark}
The numerical results indicate that the performance of our algorithm is comparable to the one in \cite{ye}. In Example \ref{haj} and
 Example \ref{myexa} we observe a slight advantage of our algorithm for some choices of the initial point. In Example \ref{linear} we note that some choices of $\delta$ give
us a different behavior. Namely, when $\delta$ is close to $0$, our algorithm requires a fewer number of iterations and less CPU time. This situation is reversed when $\delta$ is close
 to $1$. Indeed, for this case \cite[Algorithm 2.1]{ye} requires fewer iterations and less CPU time than ours. This is confirmed by the fact that for $\delta=0.5$, both algorithms have
  similar performance. This difference of behaviour for different values of  $\delta$ is due to the different linesearch used in the  algorithms. In Example \ref{ptso} the implementation is possible since for all $x\in \dom(T)$, the set $T(x)$ is a ray. Therefore the computational
 implementation of the {\bf Linesearch \ref{feasible}} is possible because the optimization problem
\begin{equation*}\label{optim} \mbox{max }  \la y,w\ra \mbox{ such that } y\in T(x), \end{equation*}
is implementable.
\end{remark}

\section{Conclusions}\label{conclu}
We have presented an algorithm for solving the Variational Inequality Problem in finite dimensional Euclidian spaces for point-to-set operators.
 We established convergence without any monotonicity assumption. Our numerical experiments showed that when the operator $T$ is point-to-point,
 our algorithm has a competitive performance when compared with similar algorithms in the literature. The {\bf Linesearch \ref{feasible}} requires the knowledge of the whole
 set $T(\alpha z+(1-\alpha)x)$. Indeed, it requires to verify that $\forall u_{\alpha}\in T(\alpha z+(1-\alpha)x)$, the inequality $\la u_{\alpha}, x-z\ra <\delta \la u,x-z\ra$ holds.
 The question of finding an implementable linesearch for the point-to-set case is an
 open problem and the subject of our future research.

\section*{Acknowledgments}
R. D\'iaz Mill\'an was partially supported by CNPq grant 200427/2015-6. This
work was concluded while the second author was visiting the  School of Information Technology and Mathematical Sciences at the  University of South
Australia. R. D\'iaz Mill\'an would like to thank the great hospitality received during his visit, particularly to Regina S. Burachik and  C. Yal\c{c}in Kaya. R. D\'iaz Mill\'an would like to extend its gratitude to Prof. Ole Peter Smith for his valuable suggestions.

\bibliographystyle{plain}

%%%%%%%%%%%%%%%%%%%%%%%%%%%%%%%%%%%%%%%%%%

%%%%%%%%%%%%%%%%%%%%%%%%%%%%%%%%%%%%%%%%%%
\end{document}